\documentclass[11pt]{article}

\usepackage{amsmath,amsthm,verbatim,amssymb,amsfonts,amscd, graphicx}
\usepackage{graphics}
\usepackage{authblk}
\usepackage{upgreek}
\usepackage{amsrefs}
\theoremstyle{plain}
\newtheorem{theorem}{Theorem}[section]

\newtheorem{proposition}[theorem]{Proposition}
\theoremstyle{remark}
\newtheorem{remark}[theorem]{Remark}
\numberwithin{equation}{section}
\newcommand{\diff}{\mathop{}\!\mathrm{d}}

\title{A survey of geometric constraints on the blowup of solutions of the Navier--Stokes equation}
\author[1]{Evan Miller}
\affil[1]{University of British Columbia, Department of Mathematics

emiller@msri.org}

\begin{document}

\maketitle

\begin{abstract}
In this survey article, we will discuss some regularity criteria for the Navier--Stokes equation that provide geometric constraints on any possible finite-time blowup. We will also discuss the physical significance of such regularity criteria.
\end{abstract}

\section{Introduction}

The Navier--Stokes equation is an evolution equation that plays a central role in fluid mechanics. While it is among the most studied partial differential equations in mathematical physics, much about its solutions---including regularity, uniqueness, and stability---remains unknown. In this survey article, we will discuss some geometric constraints on the blowup of solutions of the Navier--Stokes equation in three dimensions. 

The incompressible Navier--Stokes equation is given by
\begin{align}
    \partial_t u-\nu\Delta u+(u\cdot \nabla)u
    +\nabla p&=0 \\
    \nabla \cdot u&=0,
\end{align}
where $u\in\mathbb{R}^d$ is the velocity, $p$ is the pressure, and $\nu>0$ is the kinematic viscosity. 
The first equation expresses Newton's second law, $F=ma$, where $\partial_t u+(u\cdot\nabla)u$ gives the acceleration of a fluid particle at a given point, $-\nabla p$ describes the force due to the pressure, and $\nu\Delta u$ describes the viscous forces due to the internal friction of the fluid. 
The Euler equation for an inviscid fluid with no internal friction is obtained when $\nu=0$.
We will note that the pressure $p$ is uniquely determined by the velocity $u$, and can be eliminated using the Helmholtz decomposition, yielding
\begin{equation}
    \partial_t u-\nu\Delta u
    +P_{df}\left((u\cdot \nabla)u\right)=0.
\end{equation}
For this reason the two main types of solutions of the Navier--Stokes equation, Leray weak solutions and mild solutions, are both defined without making any reference to the pressure.

Two other crucially important objects for the study of the Navier--Stokes equation are the strain and the vorticity. The strain is the symmetric part of $\nabla u$,
\begin{equation}
    S_{ij}=\frac{1}{2}(\partial_i u_j+\partial_j u_i),
\end{equation}
while the vorticity is a vector representation of the anti-symmetric part of $\nabla u$, 
given by
\begin{equation}
    \omega=\nabla \times u.
\end{equation}

The first notion of solution for the Navier--Stokes equation was developed by Leray in his seminal work \cite{Leray}.
Leray proved the global-in-time existence of weak solutions of the Navier--Stokes equation in the sense of integrating against smooth test functions and satisfying the energy inequality for all $0<t<+\infty,$
\begin{equation} \label{EnergyIneq}
    \frac{1}{2}\|u(\cdot,t)\|_{L^2}^2
    +\nu\int_0^t\|\nabla u(\cdot,\tau)\|_{L^2}^2 \diff\tau
    \leq 
    \frac{1}{2}\left\|u^0\right\|_{L^2}^2,
\end{equation}
for all initial data $u^0\in L^2_{df}$.
We will note that smooth solutions of the Navier--Stokes equation satisfy \eqref{EnergyIneq} with equality; 
the inequality for Leray weak solutions comes from passing to weak limits for a mollified problem.
While Leray weak solutions must exist globally-in-time, they are not known to be either smooth or unique.

One way around this is the notion of mild solutions developed by Fujita and Kato \cite{KatoFujita}, which are solutions satisfying the equation
\begin{equation}
    \partial_t u -\nu\Delta u
    =-P_{df}((u\cdot\nabla)u),
\end{equation}
in the sense of convolution with the heat kernel as in Duhamel's formula.
Fujita and Kato proved the local-in-time existence of mild solutions, and furthermore that such solutions must be smooth and unique. However, it remains one of the largest open problems in the field of nonlinear PDE, indeed one of the millennium problems set by the Clay Math Institute, whether mild solutions of the Navier--Stokes equation can blowup in finite-time in three spatial dimensions (either $\mathbb{R}^3$ or $\mathbb{T}^3$). For a fuller description of this problem, see \cite{ClayNS}.

The essential problem is that the bounds from the energy equality in $L^\infty_t L^2_x$ and $L^2_t\dot{H}^1_x$ are both supercritical with respect to scaling, as the Navier--Stokes equation is invariant under the rescaling
\begin{equation}
    u^\lambda(x,t)=
    \lambda u(\lambda x,\lambda^2t),
\end{equation}
for all $\lambda>0$.
The global-in-time existence of mild solutions can be guaranteed for small initial data in certain scale critical spaces. 
Kato proved the global existence of mild solutions for small initial data in $L^3$ \cite{KatoL3}, and this was later extended to the larger scale critial space $BMO^{-1}$ by Koch and Tataru \cite{KochTataru}.

It is also known that there must be a smooth solution of the Navier--Stokes equation if there is control on the history of some scale critical norm.
Ladyzhenskaya \cite{Ladyzhenskaya}, Prodi \cite{Prodi}, and Serrin \cite{Serrin} showed that if a solution of the Navier--Stokes equation blows up in finite-time $T_{max}<+\infty$,
then for all $\frac{2}{p}+\frac{3}{q}=1, 3<q\leq+\infty$,
\begin{equation}
    \int_0^{T_{max}}\|u(\cdot,t)\|_{L^q}^p
    \diff t=+\infty.
\end{equation}
Escauriaza, Seregin, and \v{S}ver\'ak extended this result to the endpoint case $p=+\infty, q=3$ \cite{ESS}, showing that if $T_{max}<+\infty$,
\begin{equation}
    \limsup_{t\to T_{max}}\|u(\cdot,t)\|_{L^3}
    =+\infty.
\end{equation}
The limsup was later replaced by a limit by Seregin \cite{SereginLim},
and this result has also been extended to nonendpoint, scale critical Besov spaces \cites{GKP,Albritton}.

\section{Component reduction regularity criteria}

The Ladyzhenskaya-Prodi-Serrin regularity criterion has also been extended to involve regularity criteria only requiring control on certain components of $u, \nabla u, \omega,$ or $S$ in some scale-critical space.
Chae and Choe proved the first scale-critical component reduction regularity criterion \cite{ChaeChoe}, proving that if $T_{max}<+\infty$, then
\begin{equation}
    \int_0^{T_{max}} \|e_3\times 
    \omega(\cdot,t)\|_{L^q}^p \diff t
    =+\infty,
\end{equation}
for all $\frac{2}{p}+\frac{3}{q}=2, 
\frac{3}{2}<q<+\infty$.
Note that this is a regularity criterion on two vorticity components because
\begin{equation}
    e_3\times\omega=(-\omega_2,\omega_1,0).
\end{equation}
The endpoint case $q=\frac{3}{2}$ remains an interesting open question: does $T_{max}<+\infty$ imply that
\begin{equation}
    \limsup_{t\to T_{max}}
    \|e_3\times \omega\|_{L^\frac{3}{2}}
    =+\infty.
\end{equation}
This regularity criterion has been extended to non-endpoint Besov spaces by Chen and Zhang \cite{ChenZhangBesov} and to endpoint Besov spaces by Guo, Ku\v{c}era and Skal\'ak \cite{GuoBesov}.
There is also related work by the author requiring global regularity for solutions of the Navier--Stokes equation when the initial data satisfies a condition requiring $e_3\times \omega^0$ to be sufficiently small \cite{MillerAlmost2D}.

Kukavica and Ziane proved a scale-critical component-reduction regularity criterion involving the derivative in just one direction \cite{KukavicaZiane}, 
proving that if a solution of the Navier--Stokes equation blows up in finite-time $T_{max}<+\infty$, then for all $\frac{2}{p}+\frac{3}{q}=2, 
\frac{9}{4}\leq q \leq 3$,
\begin{equation}
    \int_0^{T_{max}}\|\partial_3 u(\cdot,t)\|_{L^q}^p
    \diff t=+\infty.
\end{equation}
Kukavica, Rusin, and Ziane recently extended this to a localized regularity criterion \cites{KukavicaLocal}.

There are also a number of papers which extend the range of exponents for which this result holds.
Cao extended \cite{Cao} this result to the range of exponents $\frac{27}{16}\leq q\leq 3$, although the proof in this paper only covers the range $\frac{27}{16}\leq q\leq \frac{5}{2}$, with the rest of the range already proven by Kukavica and Ziane in \cite{KukavicaZiane}.
Zhang extended \cite{ZujinZhang} the range of exponents to include $\frac{3\sqrt{37}}{4}-3\leq q \leq 3$,
and Namlyeyeva and Skal\'ak then extended the lower bound on this range further in \cite{Namlyeyeva}, although still not to the endpoint case $q=\frac{3}{2}$.
Finally, Skal\'ak extended \cites{Skalak1,Skalak2} this result to include the range $\frac{3}{2}<q\leq \frac{19}{6}$.
The extension in \cite{Skalak1} is particularly important because it includes up until the endpoint case $q=\frac{3}{2}$. The endpoint case itself remains open. In particular, it is not known whether $T_{max}<+\infty$ implies that
\begin{equation}
    \limsup_{t\to T_{max}}
    \left\|\partial_3 u(\cdot,t)
    \right\|_{L^\frac{3}{2}}
    =+\infty.
\end{equation}
Very recently, Chen, Fang and Zhang \cite{CFZ} and Giang and Khai \cite{GiangKhai} extended the range of exponents to $3\leq q\leq 6$ using different methods, which brings the range of exponents for which there is a scale critical regularity criterion in terms of one directional derivative of velocity to $\frac{3}{2}<q\leq 6$, with the endpoint case in particular remaining open.

Another recent regularity result in terms of the unidirectional derivative is a global regularity result for a certain class of initial data. Liu and Zhang proved global regularity for initial data with a small unidirectional derivative, requiring control only on the initial data, not the history of the solution \cite{LiuZhang}. This was further generalized to the anisotropic Navier--Stokes system, with dissipation in only the horizontal directions in \cite{UnidirectionalAniso}.

Another approach to component reduction regularity criteria is based on proving regularity criteria involving only one component of the velocity.
Chemin and Zhang proved \cite{CheminZhang} that if a smooth solution of the Navier--Stokes equation blows up in finite-time $T_{max}<+\infty$, then for all $4<p<6$,
\begin{equation}
    \int_0^{T_{max}}\left\|u_3(\cdot,t)
    \right\|_{\dot{H}^{\frac{1}{2}+\frac{2}{p}}}^p
    \diff t=+\infty.
\end{equation}
Chemin, Zhang, and Zhang then extended \cite{CheminZhangZhang} this result to the range $4<p<+\infty$,
and Han, Lei, Li, and Zhao extended \cite{HanLeiLiZhao} the result further to the range $2\leq p<+\infty$.
Neustupa, Novotn\'y and Penel had proven an early regularity criterion involving only $u_3$ that was among the first component-reduction regularity criteria \cite{NeustupaOneComp}, but their result was subcritical in terms of scaling.

There are no scale-critical regularity criteria in terms of just one entry of $\nabla u$, but there are a number of results giving subcritical regularity criteria in terms of just one diagonal entry $\partial_i u_i$  or one non-diagonal entry $\partial_i u_j$
\cites{A,B,C,D,E,F,I,J,K,HanLeiLiZhao,SkalakRecent}.
The regularity criteria in terms of one diagonal entry of $\nabla u$ are closer to being scale-critical than the regularity criteria in terms of just one non-diagonal entry.

There are also component reduction regularity criteria in terms of the eigenvalues of the strain matrix. Let $\lambda_1(x,t)\leq \lambda_2(x,t) \leq \lambda_3(x,t)$ be the eigenvalues of $S(x,t)$
and let $\lambda_2^+=\max\left(0,\lambda_2\right)$.
Neustupa and Penel proved \cites{NeuPen1,NeuPen2} that if $T_{max}<+\infty$, then for all $\frac{2}{p}+\frac{3}{q}=2, 
\frac{3}{2}<q\leq +\infty$,
\begin{equation}
    \int_0^{T_{max}}
    \left\|\lambda_2^+(\cdot,t)\right\|_{L^q}^p
    \diff t=+\infty.
\end{equation}
This was later proven independently by the author in \cite{MillerStrain} using somewhat different methods involving the evolution equation for the strain.
This result was also generalized to a localized regularity criterion in \cites{NeuPen3}.
As a corollary the author proved that the strain must blow up in all directions \cite{MillerStrain}. In particular, the author proved that for any unit vector valued function, 
$v\in L^\infty\left(
[0,T_{max}]\times\mathbb{R}^3\right), |v(x,t)|=1$, almost everywhere,
if $T_{max}<+\infty$, then for all 
$\frac{2}{p}+\frac{3}{q}=2, 
\frac{3}{2}<q\leq +\infty$,
\begin{equation}
    \int_0^{T_{max}}
    \left\|(Sv)(\cdot,t)\right\|_{L^q}^p
    \diff t=+\infty.
\end{equation}
Interestingly, the special case of this corollary where $v=e_3$ is actually equivalent to the regularity criterion proven by Chae and Choe on two vorticity components, as we will now show.

\begin{proposition} \label{Equiv}
For all $1<q<+\infty$, and for all $u\in W^{1,q}, \nabla\cdot u=0$,
\begin{align}
    \|\nabla u_3\|_{L^q}+\|\partial_3 u\|_{L^q}
    &\leq C_q \|2Se_3\|_{L^q} \\
    \|\nabla u_3\|_{L^q}+\|\partial_3 u\|_{L^q}
    &\leq C_q \|e_3\times \omega\|_{L^q},
\end{align}
and furthermore
\begin{equation}
    \frac{1}{C_q}\|e_3\times \omega\|_{L^q}
    \leq \|2Se_3\|_{L^q} \leq C_q
    \|e_3\times \omega\|_{L^q}.
\end{equation}
\end{proposition}
\begin{proof}
Begin by observing that
\begin{align}
    2S e_3 &=\nabla u_3+\partial_3 u \\
    e_3\times \omega&= \nabla u_3-\partial_3 u.
\end{align}
Observing that $\nabla \cdot \partial_3 u=0$ and that $\nabla u_3$ is a gradient, we can apply the Helmholtz decomposition and see that
\begin{equation}
    \nabla u_3=P_{gr}(2Se_3)
    =P_{gr}(e_3\times\omega)
\end{equation}
and that
\begin{equation}
    \partial_3 u=P_{df}(2Se_3)
    =-P_{df}(e_3\times\omega).
\end{equation}
Using the boundedness of the Helmholtz decomposition from $L^q$ to $L^q$,
we find that for all $1<q<+\infty$,
\begin{align}
    \|\nabla u_3\|_{L^q}+\|\partial_3 u\|_{L^q}
    &\leq C_q \|2Se_3\|_{L^q} \\
    \|\nabla u_3\|_{L^q}+\|\partial_3 u\|_{L^q}
    &\leq C_q \|e_3\times \omega\|_{L^q}.
\end{align}
Applying the triangle inequality this also implies that for all $1<q<+\infty$,
\begin{align}
    \|e_3\times \omega\|_{L^q}
    &\leq
    \|\nabla u_3\|_{L^q}+\|\partial_3 u\|_{L^q}\\
    &\leq C_q \|2Se_3\|_{L^q},
\end{align}
and
\begin{align}
    \|2Se_3\|_{L^q}
    &\leq
    \|\nabla u_3\|_{L^q}+\|\partial_3 u\|_{L^q}\\
    &\leq C_q \|e_3\times \omega\|_{L^q},
\end{align}
and so this completes the proof.
\end{proof}

\begin{remark}
This result implies that Chae and Choe's regularity criterion on $e_3\times \omega$ is equivalent to the regularity criterion on $Se_3$, which is a special case of the result that the strain must blowup in all directions. It also implies that the regularity criterion in terms of two vorticity components is strictly weaker than the regularity criteria for $u_3 \in L^p_t \dot{H}^{\frac{1}{2}+\frac{2}{p}}$ for $4\leq p<+\infty$, using the Sobolev embedding
\begin{equation}
    \|u_3\|_{\dot{H}^{\frac{1}{2}+\frac{2}{p}}}
    \leq C_q \|\nabla u_3\|_{L^q},
\end{equation}
where $\frac{2}{p}+\frac{3}{q}=2$.
Finally, Proposition \ref{Equiv} implies that the regularity criterion for 
$e_3\times \omega\in L^p_t L^q_x$ is strictly weaker than the regularity criterion on $\partial_3 u\in L^p_tL^q_x$ for the range of exponents $\frac{3}{2}<q\leq 6$.
\end{remark}

The regularity criterion for $Sv \in L^p_t L^q_x$ holds for generic unit vector allowed to vary in space, so it is natural to ask if the vector can also be allowed to vary in the regularity criterion for $v\times \omega \in L^p_tL^q_x$. Of course, by rotational symmetry any fixed unit vector $v$ can replace $e_3$, but can the vector also be allowed to vary in space? The author provided a positive answer to this question for the case $q=2, p=4$ in \cite{MillerAnisoVort}, showing that if $T_{max}<+\infty$, then for all 
$v\in L^\infty\left(
[0,T_{max}]\times\mathbb{R}^3\right)$, such that $|v(x,t)|=1$ almost everywhere, 
$\nabla v \in L^\infty\left(
[0,T_{max}]\times \mathbb{R}^3\right)$,
\begin{equation} \label{LocalAnisoVort}
    \int_0^{T_{max}}\|(v\times\omega)(\cdot,t)
    \|_{L^2}^4 \diff t=+\infty.
\end{equation}

In addition to these component-reduction-type regularity criteria, an approach to regularity criteria for the Navier--Stokes equation with an even more explicitly geometric flavour involves the direction of vorticity.
Constantin and Fefferman proved that the the vorticity direction must vary rapidly in regions of large vorticity in order for blowup to occur \cite{ConstantinFefferman}.
In particular they showed that if a solution of the Navier--Stokes equation blows up in finite-time $T_{max}<+\infty$, then for all $R>0$
\begin{equation}
    \sup_{\substack{
    |\omega(x,t)|,|\omega(y,t)|>R \\
    x\neq y}} 
    \frac{|\eta(x,t)\times\eta(y,t)|}{|x-y|}
    =+\infty,
\end{equation}
where $\eta=\frac{\omega}{|\omega|}$.
This was generalized by Beir\~ao da Veiga and Berselli in \cite{daVeigaBerselli}.
As a corollary of their generalized result, they also proved a regularity criterion in terms of the gradient of the vorticity direction, showing that if $T_{max}<+\infty$, for all $\frac{2}{p}+\frac{3}{q}=\frac{1}{2},
6\leq q \leq +\infty$,
\begin{equation}
    \int_0^{T_{max}}
    \|\nabla \eta(\cdot,t)\|_{L^q}^p 
    \diff t
    =+\infty.
\end{equation}

The regularity criterion involving control on 
$v\times \omega \in L^4_tL^2_x$, where the unit vector $v$ may vary in time and space can be seen as interpolating between Chae and Choe's regularity criterion on two vorticity components and Beir\~ao da Veiga and Berselli's regularity criterion on the gradient of the vorticity direction, although the interpolation is suboptimal in terms of scaling at the later endpoint. The case $v=e_3$---or any fixed unit vector---corresponds to Chae and Choe's regularity criterion, while the case $v=\eta$ corresponds to a version for a version of Beir\~ao da Veiga and Berselli's regularity criterion that is weaker in terms of scaling.
For a more detailed discussion, see \cite{MillerAnisoVort}.

Skal\'ak recently weakened the assumptions on the regularity criterion involving $v\times\omega$, both expanding the range of exponents to $\frac{3}{2}<q<+\infty$ rather than just $q=2$ and relaxing the regularity assumption on the unit vector $v$. In this paper \cite{SkalakVort}, the interpolation between Chae and Choe's regularity criterion on two vorticity components and Beir\~ao da Veiga and Berselli's regularity criterion on the gradient of vorticity direction is optimal.

\section{Physical interpretations}

The Navier--Stokes equation has globally smooth solutions in two dimensions, so the component-reduction regularity criteria for the Navier--Stokes equation can be seen as perturbative conditions requiring regularity if a solution of the three dimensional Navier--Stokes equation is close enough to being two dimensional in some sense. If we have control on $\partial_3 u$ or $u_3$ in a scale critical space, then the solution is close enough to being two dimensional that it must be globally regular.
These component reduction regularity criteria show that blowup has to be fully three dimensional globally.
For two dimensional flows in the $xy$ plane, the vorticity is entirely in the $z$ direction, so control on $e_3\times \omega$, the vorticity in the $xy$ plane, can likewise be seen as saying that if the solution is close enough to being two dimensional in a scale-critical space, then it must be smooth. We should note that there is nothing special about the $z$ direction, and the rotational invariance of the Navier--Stokes equation means this direction can be replaced with any fixed direction and each of the above results still hold.

The regularity criterion involving $\lambda_2^+$ also has a clear physical interpretation. Like the other component reduction regularity criteria, it requires that blowup must be fully three dimensional, but it is stronger in that it requires blowup to be locally three dimensional, and gives a specific geometric structure---planar stretching and axial compression. The most natural example of this structure is two colliding jets, which leads the axial compression in the direction of the jets, and planar stretching in the plane perpendicular to the jets.

The regularity criteria involving the direction of the vorticity are also extremely significant physically, because the rapid change in the orientation of vortices has been understood, at least at a heuristic level, to be a fundamental feature of turbulence for at least half a millennium, going back to Leonardo da Vinci's ``Studies in Turbulent Flow'' (see Figure 1) \cite{NatureDaVinci}. This phenomenological feature of turbulence was given an analytical expression in the theory of Navier--Stokes regularity criteria by Constantin and Fefferman \cite{ConstantinFefferman} and later extended in \cite{daVeigaBerselli}.
This feature of turbulence is also expressed in the Kolmogorov-Obhukov theory of turbulent energy cascade, which not only gives a scaling law for the transfer of energy to shorter length scales (and equivalently higher order Fourier modes), but also suggests that turbulence must be anisotropic at the smallest length scales in the inertial range \cites{Kolmogorov,Obukhov}. This means in particular that the vorticity does not have a preferred direction at the smallest length scales.

\begin{center}
\begin{figure*}[ht]
    \centering
    \includegraphics[width=0.90\textwidth]{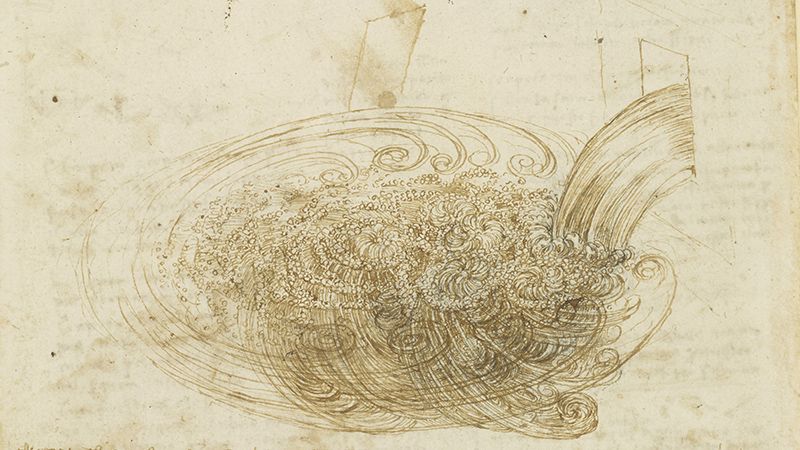}
    \caption{Leonardo da Vinci, \textit{Studies of Turbulent Water}}
\end{figure*}
\end{center}

The regularity criterion for $v\times \omega\in L^4_tL^2_x$, where $v$ is a unit vector with a bounded gradient also reflects this aspect of turbulent flow. Whereas the regularity criterion on the vorticity restricted to a fixed plane requires that the vorticity must become unbounded in every fixed plane in order for blowup to occur, and that blowup for the vorticity in this sense must be globally three dimensional, the regularity criterion on $v\times \omega \in L^4_t L^2_x$ requires that the vorticity must become unbounded restricted to any plane that may vary in space so long as the gradient of the unit vector orthogonal to the plane remains bounded. This requires the structure of the vorticity to be locally three dimensional. The understanding of turbulent flow going back to da Vinci involves both a rapid change in vorticity direction (the orientation of vortices in physical terms) and the use of all available degrees of freedom, with the vorticity pointing in every which way in the turbulent region. The locally anisotropic regularity criterion in terms of vorticity proven by the author in \cite{MillerAnisoVort} is consistent with this phenomenological picture of turbulence.

While the geometric constraints discussed here give a description of the features of possible blowup solutions, if they do in fact exist,
on their own they are not enough to guarantee global regularity. There are model equations involving the relevant constraint spaces that respect these geometric constraints on blowup, and nonetheless exhibit finite-time blowup. For instance, the author proved the existence of finite-time blowup for a model equation for the self-amplification of strain,
\begin{equation}
    \partial_t S -\nu\Delta S
    +\frac{2}{3}P_{st}\left(S^2\right)=0,
\end{equation}
that respects of number of these constraints
including on $\lambda_2^+ \in L^p_t L^q_x$ 
and $e_3 \times \omega \in L^p_t L^q_x$
\cites{MillerStrainModel}.
It is very unlikely, therefore, that simple technical improvements or extensions of these geometric regularity criteria will be enough to guarantee global regularity. The question of global regularity (or, at this point perhaps equally likely, finite-time blowup) is likely only to be resolved through a significantly improved understanding of the possible mechanisms for the depletion of nonlinearity.

\section*{Acknowledgements}
This publication was supported in part by the Fields Institute for Research in the Mathematical Sciences while the author was in residence during the Fall 2020 semester. Its contents are solely the responsibility of the author and do not necessarily represent the official views of the Fields Institute.
This material is based upon work supported by the National Science Foundation under Grant No. DMS-1440140 while the author participated in a program that was hosted by the Mathematical Sciences Research Institute in Berkeley, California, during the Spring 2021 semester.
The author would like to thank Rapha\"el Danchin, Reinhard Farwig, Sarka Necasova, Ji\v{r}\'{i} Neustupa, and the staff of the Centre International de Rencontres Mathématiques (CIRM) for all of the work they did to make the conference Vorticity, Rotation and Symmetry (V) – Global Results and Nonlocal Phenomena a success in spite of the virtual format and the limitations due to the pandemic.

\section*{Compliance with ethical standards}
Conflict of interest: the author declares that he has no conflict of interest.
\bibliographystyle{plain}
\bibliography{Bib}

\end{document}